\documentclass[12pt]{amsproc}
\usepackage{hyperref}
\usepackage{geometry}
\usepackage{graphicx}
\usepackage{amsthm}
\usepackage{amssymb}

\theoremstyle{plain}
\newtheorem{thm}{Theorem}[section]
\newtheorem{cor}[thm]{Corollary}
\newtheorem{lem}[thm]{Lemma}
\newtheorem{prop}[thm]{Proposition}
\newtheorem{defn}[thm]{Definition}
\newtheorem{exa}[thm]{Example}

\newtheorem{question}[thm]{Question}

\begin{document}

\title{Graded Weakly $2$-Absorbing Ideals over Non-Commutative Rings}

\author{Azzh Saad \textsc{Alshehry}}
\address{Department of Mathematical Sciences, Faculty of Sciences, Princess Nourah Bint Abdulrahman University, P.O. Box 84428, Riyadh 11671, Saudi Arabia}
\email{asalshihry@pnu.edu.sa}

\author{Jebrel \textsc{Habeb}}
\address{Department of Mathematics, Yarmouk University, Irbid, Jordan}
\email{jhabeb@yu.edu.jo}

\author{Rashid \textsc{Abu-Dawwas}}
\address{Department of Mathematics, Yarmouk University, Irbid, Jordan}
\email{rrashid@yu.edu.jo}

\author{Ahmad \textsc{Alrawabdeh}}
\address{Department of Mathematics, Yarmouk University, Irbid, Jordan}
\email{ahmad.rwabdeh@gmail.com}



\begin{abstract}
For commutative graded rings, the concept of graded $2$-absorbing (graded weakly $2$-absorbing) ideals was introduced and examined by Al-Zoubi, Abu-Dawwas and \c{C}eken (Hacettepe Journal of Mathematics and Statistics, 48 (3) (2019), 724-731) as a generalization of the concept of graded prime (graded weakly prime) ideals. Up to now, research on these topics mainly concentrated on commutative graded rings. On the other hand, graded prime ideals over non-commutative graded rings have been introduced and examined by Abu-Dawwas, Bataineh and Al-Muanger (Vietnam Journal of Mathematics, 46 (3) (2018), 681-692). As a generalization of graded prime ideals over non-commutative graded rings, the concept of graded $2$-absorbing ideals over non-commutative graded rings has been introduced and investigated by Abu-Dawwas, Shashan and Dagher (WSEAS Transactions on Mathematics, 19 (2020), 232-238). Recently, graded weakly prime ideals over non-commutative graded rings have been introduced and studied by Alshehry and Abu-Dawwas (Communications in Algebra, 49 (11) (2021), 4712-4723). In this article, we introduce and study the concept of graded weakly $2$-absorbing ideals as a generalization of graded weakly prime ideals in a non-commutative graded ring, and show that many of the results in commutative graded rings also hold in non-commutative graded rings.
\end{abstract}

\maketitle

\section{Introduction}

Throughout this article, the rings are associative but not necessarily assumed to have unity unless indicated otherwise. Also an ideal means a two-sided ideal. Let $G$ be a group with identity $e$ and $R$ be a ring. Then $R$ is called $G$-graded if $R=\displaystyle\bigoplus_{g\in G} R_{g}$ with $R_{g}R_{h}\subseteq R_{gh}$ for all $g, h\in G$, where $R_{g}$ is an additive subgroup of $R$ for all $g\in G$. The elements of $R_{g}$ are called homogeneous of degree $g$. If $a\in R$, then $a$ can be written uniquely as $\displaystyle\sum_{g\in G}a_{g}$, where $a_{g}$ is the component of $a$ in $R_{g}$. The component $R_{e}$ is a subring of $R$ and if $R$ has a unity $1$, then $1\in R_{e}$. The set of all homogeneous elements of $R$ is $h(R)=\displaystyle\bigcup_{g\in G}R_{g}$. Let $P$ be an ideal of a graded ring $R$. Then $P$ is called a graded ideal if $P=\displaystyle\bigoplus_{g\in G}(P\cap R_{g})$, i.e., for $a\in P$, $a=\displaystyle\sum_{g\in G}a_{g}$ where $a_{g}\in P$ for all $g\in G$. It is not necessary that every ideal of a graded ring is a graded ideal (\cite{Alshehry Dawwas}, Example 1.1). For more details and terminology, see \cite{Hazart, Nastasescue}.

For commutative graded rings, graded $2$-absorbing ideals, which are a generalization of graded prime ideals, were introduced and investigated in \cite{Zoubi Dawwas Ceken}. Recall from \cite{Atani} that a proper graded ideal $P$ of a commutative graded ring $R$ is said to be a graded weakly prime ideal of $R$ if whenever $x, y\in h(R)$ and $0\neq xy\in P$, then either $x\in P$ or $y\in P$. Also from \cite{Zoubi Dawwas Ceken} a proper graded ideal $P$ of a commutative
graded ring $R$ is called a graded $2$-absorbing ideal of $R$ if whenever $x, y, z\in h(R)$ and $xyz\in P$, then either $xy\in P$ or $xz\in P$ or $yz\in P$. The concept of a graded weakly $2$-absorbing ideal of a commutative graded ring $R$ was introduced in \cite{Zoubi Dawwas Ceken}. A proper graded ideal $P$ of a commutative graded ring $R$ is called a graded weakly $2$-absorbing ideal of $R$ if whenever $x, y, z\in h(R)$ and $0\neq xyz\in P$, then either $xy\in P$ or $xz\in P$ or $yz\in P$.

Graded prime ideals over non-commutative graded rings have been introduced and examined by Abu-Dawwas, Bataineh and Al-Muanger in \cite{Dawwas Bataineh Muanger}. A proper graded ideal $P$ of $R$ is said to be graded prime if whenever $I$ and $J$ are graded ideals of $R$ such that $IJ\subseteq P$, then $I\subseteq P$ or $J\subseteq P$. As a generalization of graded prime ideals over non-commutative graded rings, the concept of graded $2$-absorbing ideals over non-commutative graded rings has been introduced and investigated by Abu-Dawwas, Shashan and Dagher in \cite{Dawwas Shashan Dagher}. A proper graded ideal $P$ of $R$ is said to be graded $2$-absorbing if whenever $x, y, z\in h(R)$ such that $xRyRz\subseteq P$, then $xy\in P$ or $yz\in P$ or $xz\in P$. Recently, graded weakly prime ideals over non-commutative graded rings have been introduced and studied by Alshehry and Abu-Dawwas in \cite{Alshehry Dawwas}. A proper graded ideal $P$ of $R$ is said to be graded weakly prime if whenever $I$ and $J$ are graded ideals of $R$ such that $0\neq IJ\subseteq P$, then $I\subseteq P$ or $J\subseteq P$. In this article, we are following \cite{Groenewald} to introduce and study the concept of graded weakly $2$-absorbing ideals as a generalization of graded weakly prime ideals in a non-commutative graded ring, and show that many of the results in commutative graded rings also hold in non-commutative graded rings.

\section{Graded Weakly $2$-Absorbing Ideals}

In this section, we introduce and study the concept of graded weakly $2$-absorbing ideals. In what follows a ring $R$ is a non-commutative ring with unity unless indicated otherwise.

\begin{defn} Let $P$ be a proper graded ideal of a graded ring $R$. Then $P$ is said to be a graded weakly $2$-absorbing ideal of $R$ if $0\neq xRyRz\subseteq P$ implies $xy\in P$ or $yz\in P$ or $xz\in P$ for all $x, y, z\in h(R)$. If $0\neq xyz\in P$ implies $xy\in P$ or $yz\in P$ or $xz\in P$ for all $x, y, z\in h(R)$, we say that $P$ is graded completely weakly $2$-absorbing.
\end{defn}

Apparently, if $R$ is a commutative graded ring with unity, then the two concepts of graded weakly $2$-absorbing and graded completely weakly $2$-absorbing
coincide. The next example shows that this will not be the case for non-commutative graded rings:

\begin{exa}\label{Remark 4.6} Consider $R=M_{2}(\mathbb{Z})$ (the ring of all $2\times 2$ matrices with integer entries) and $G=\mathbb{Z}_{4}$. Then $R$ is $G$-graded by $R_{0}=\left(\begin{array}{cc}
                             \mathbb{Z} & 0 \\
                             0 & \mathbb{Z}
                           \end{array}\right)$, $R_{2}=\left(\begin{array}{cc}
                             0 & \mathbb{Z} \\
                             \mathbb{Z} & 0
                           \end{array}\right)$ and $R_{1}=R_{3}=0$. Consider the graded ideal $P=M_{2}(2\mathbb{Z})$ of $R$. Clearly, $P$ is a graded prime ideal of $R$ and hence also a graded weakly $2$-absorbing ideal of $R$. On the other hand, $P$ is not a graded completely weakly $2$-absorbing ideal of $R$ since $A=\left(\begin{array}{cc}
                             3 & 0 \\
                             0 & 2
                           \end{array}\right)$, $B=\left(\begin{array}{cc}
                             0 & 3 \\
                             5 & 0
                           \end{array}\right)$ and $C=\left(\begin{array}{cc}
                             7 & 0 \\
                             0 & 4
                           \end{array}\right)\in h(R)$ with $0\neq ABC=\left(\begin{array}{cc}
                             0 & 36 \\
                             70 & 0
                           \end{array}\right)\in P$, $AB=\left(\begin{array}{cc}
                             0 & 9 \\
                             10 & 0
                           \end{array}\right)\notin P$, $AC=\left(\begin{array}{cc}
                             21 & 0 \\
                             0 & 8
                           \end{array}\right)\notin P$ and $BC=\left(\begin{array}{cc}
                             0 & 12 \\
                             35 & 0
                           \end{array}\right)\notin P$.
\end{exa}

Undoubtedly, every graded $2$-absorbing ideal of a graded ring $R$ is a graded weakly $2$-absorbing ideal of $R$. If $R$ is any graded ring, then $P=\{0\}$ is a graded weakly $2$-absorbing ideal of $R$ by definition. On the other hand, $P=\{0\}$ is not necessary to be a graded $2$-absorbing ideal, check the following example:

\begin{exa} Consider $R=\mathbb{Z}_{8}[i]$ and $G=\mathbb{Z}_{2}$. Then $R$ is $G$-graded by $R_{0}=\mathbb{Z}_{8}$ and $R_{1}=i\mathbb{Z}_{8}$. Undeniably, $P=\{0\}$ is not a graded $2$-absorbing ideal of $R$ since $2\in h(R)$ with $2R2R2\subseteq P$ but $2.2\notin P$.
\end{exa}

\begin{lem}\label{1} Let $P$ be a graded ideal in a graded ring $R$. Assume that $P$ is a graded weakly prime ideal of $R$. If $I$ and $J$ are graded right (left) ideals of $R$ such that $0\neq IJ\subseteq P$, then $I\subseteq P$ or $J\subseteq P$.
\end{lem}

\begin{proof} Suppose that $I$ and $J$ are graded right (left) ideals of $R$ such that $0\neq IJ\subseteq P$. Let $(I)$ and $(J)$ be the graded ideals generated by $I$ and $J$ respectively. Then $0\neq(I)(J)\subseteq P$, whence $I\subseteq(I)\subseteq P$ or $J\subseteq(J)\subseteq P$.
\end{proof}

\begin{lem}\label{Lemma 1.2} Let $P$ be a graded ideal in a graded ring $R$. If $P$ is a graded weakly prime ideal of $R$ and $0\neq xRyRz\subseteq P$ for $x, y, z\in h(R)$, then $x\in P$ or $y\in P$ or $z\in P$.
\end{lem}

\begin{proof} Suppose that $x, y, z\in h(R)$ such that $0\neq xRyRz\subseteq P$. Then $0\neq (Rx)RyRz\subseteq P$ and it follows from Lemma \ref{1} that $x\in Rx\subseteq P$ or $0\neq RyRz\subseteq P$. By reiterating this, the result follows.
\end{proof}

\begin{prop}\label{Proposition 4.2} Let R be a graded ring and $P$ be a proper graded ideal of $R$. If $P$ is a graded weakly prime ideal of $R$,
then $P$ is a graded weakly $2$-absorbing ideal of $R$.
\end{prop}

\begin{proof} Let $x, y, z\in h(R)$ such that $0\neq xRyRz\subseteq P$. By Lemma \ref{Lemma 1.2}, we have $x\in P$ or $y\in P$ or $z\in P$. Accordingly, $xy\in P$ or $yz\in P$ or $xz\in P$, and the result holds.
\end{proof}

\begin{prop}\label{Corollary 4.4} If $P$ and $K$ are two distinct graded weakly prime ideals of a graded ring $R$, then $P\bigcap K$ is a graded weakly $2$-absorbing ideal of $R$.
\end{prop}

\begin{proof} If $P\bigcap K=\{0\}$, then it is clear that $P\bigcap K$ is a graded weakly $2$-absorbing ideal of $R$. Let $x_{1}, x_{2}, x_{3}\in h(R)$ such that $0\neq x_{1}Rx_{2}Rx_{3}\subseteq P\bigcap K$. Then $0\neq x_{1}Rx_{2}Rx_{3}\subseteq P$ and $0\neq x_{1}Rx_{2}Rx_{3}\subseteq K$. It now follows from Lemma \ref{Lemma 1.2} that $x_{i}\in P$ and $x_{j}\in K$ for some $i, j$, and then $x_{i}x_{j}\in P\bigcap K$. Thereupon, $P\bigcap K$ is a graded weakly $2$-absorbing ideal of $R$.
\end{proof}

If $R$ and $T$ are two $G$-graded rings, then $R\times T$ is $G$-graded by $(R\times T)_{g}=R_{g}\times T_{g}$ for all $g\in G$. By (\cite{Bataineh Dawwas}, Lemma 4), $P\times K$ is a graded ideal of $R\times T$ if and only if $P$ is a graded ideal of $R$ and $K$ is a graded ideal of $T$. The next example shows that there are graded weakly $2$-absorbing ideals which are not graded weakly prime:

\begin{exa}\label{Remark 4.5} Consider $R=\mathbb{Z}_{2}[i]$, $T=\mathbb{Z}_{4}[i]$ and $G=\mathbb{Z}_{2}$. Then $R$ is $G$-graded by $R_{0}=\mathbb{Z}_{2}$ and $R_{1}=i\mathbb{Z}_{2}$. Also, $T$ is $G$-graded by $T_{0}=\mathbb{Z}_{4}$ and $T_{1}=i\mathbb{Z}_{4}$. So, $R\times T$ is $G$-graded by $(R\times T)_{j}=R_{j}\times T_{j}$ for all $j=0, 1$. Now, $\{0\}$ is a graded ideal of $R$ and $2T$ is a graded ideal of $T$ as $2\in h(T)$, so $P=\{0\}\times 2T$ is a graded ideal of $R\times T$ by (\cite{Bataineh Dawwas}, Lemma 4). Indeed, $P$ is not a graded weakly prime ideal of $R\times T$ since $x=(0, 1), y=(1, 2)\in h(R\times T)$ with $(0, 0)\neq xy=(0, 2)\in P$, $x\notin P$ and $y\notin P$. On the other hand, $P$ is graded $2$-absorbing and hence a graded weakly $2$-absorbing ideal of $R\times T$.
\end{exa}

\begin{thm}\label{Theorem 4.7} Let $P$ be a proper graded ideal of a graded ring $R$. Suppose that for graded left ideals $A, B$ and $C$ of $R$ such that $0\neq ABC\subseteq P$, we have $AC\subseteq P$ or $BC\subseteq P$ or $AB\subseteq P$. Then $P$ is a graded weakly $2$-absorbing ideal of $R$.
\end{thm}

\begin{proof} Let $x, y, z\in h(R)$ such that $0\neq xRyRz\subseteq P$. Then $RxRyRzR\subseteq P$, and as a consequence, since $R$ has a unity, $0\neq xRyRz=1.xR.1.yR.1.z.1\subseteq(RxR)(RyR)(RzR)\subseteq P$. By assumption, we have $xy\in(RxR)(RyR)\subseteq P$ or $yz\in(RyR)(RzR)\subseteq P$ or $xz\in(RxR)(RzR)\subseteq P$. Thus, $P$ is a graded weakly $2$-absorbing ideal of $R$.
\end{proof}

If $R$ is a $G$-graded ring and $K$ is a graded ideal of $R$, then $R/K$ is $G$-graded by $(R/K)_{g}=(R_{g}+K)/K$ for all $g\in G$. By (\cite{Saber}, Lemma 3.2), if $K$ is a graded ideal of $R$ and $P$ is an ideal of $R$ such that $K\subseteq P$, then $P$ is a graded ideal of $R$ if and only if $P/K$ is a graded ideal of $R/K$.

\begin{prop}\label{Proposition 4.8} Let $R$ be a graded ring and $P$ be a graded weakly $2$-absorbing ideal of $R$. If $K$ is a graded ideal of $R$ with $K\subseteq P$, then $P/K$ is a graded weakly $2$-absorbing ideal of $R/K$.
\end{prop}

\begin{proof} Let $x+K, y+K, z+K\in h(R/K)$ such that $0+K\neq(x+K)(R/K)(y+K)(R/K)(z+K)\subseteq P/K$. Then $x, y, z\in h(R)$ such that $0\neq xRyRz\subseteq P$. As $P$ is a graded weakly $2$-absorbing ideal of $R$, we have $xy\in P$ or $yz\in P$ or $xz\in P$, and hence $(x+K)(y+K)\in P/K$ or $(y+K)(z+K)\in P/K$ or $(x+K)(z+K)\in P/K$. So, $P/K$ is a graded weakly $2$-absorbing ideal of $R/K$.
\end{proof}

\begin{prop}\label{Proposition 4.9} Let $K\subseteq P$ be proper graded ideals of a graded ring $R$. If $K$ is a graded weakly $2$-absorbing ideal of $R$ and $P/K$ is a graded weakly $2$-absorbing ideal of $R/K$, then $P$ is a graded weakly $2$-absorbing ideal of $R$.
\end{prop}

\begin{proof} Let $x, y, z\in h(R)$ such that $0\neq xRyRz\subseteq P$. Then $x+K, y+K, z+K\in h(R/K)$ such that $(x+K)(R/K)(y+K)(R/K)(z+K)\subseteq P/K$. If $0\neq xRyRz\subseteq K$, then $xy\in K\subseteq P$ or $yz\in K\subseteq P$ or $xz\in K\subseteq P$ since $K$ is a graded weakly $2$-absorbing ideal of $R$. If $xRyRz\nsubseteq K$, then $0+K\neq(x+K)(R/K)(y+K)(R/K)(z+K)\subseteq P/K$. Since $P/K$ is a graded weakly $2$-absorbing ideal of $R/K$, $(x+K)(y+K)\in P/K$ or $(y+K)(z+K)\in P/K$ or $(x+K)(z+K)\in P/K$, which implies that $xy\in P$ or $yz\in P$ or $xz\in P$. Therefore, $P$ is a graded weakly $2$-absorbing ideal of $R$.
\end{proof}

Let $R$ and $T$ be two $G$-graded rings. In \cite{Nastasescue}, a ring homomorphism $f:R\rightarrow T$ is said to be a graded homomorphism if $f(R_{g})\subseteq T_{g}$ for all $g\in G$.

\begin{prop} Let $R$ and $T$ be two $G$-rings and $f:R\rightarrow T$ be a graded homomorphism. Then $Ker(f)$ is a graded ideal of $R$.
\end{prop}

\begin{proof} Clearly, $Ker(f)$ is an ideal of $R$. Let $x\in Ker(f)$. Then $x\in R$ such that $f(x)=0$. Now, $x=\displaystyle\sum_{g\in G}x_{g}$, where $x_{g}\in R_{g}$ for all $g\in G$, which implies that $f(x_{g})\in f(R_{g})\subseteq T_{g}$ for all $g\in G$. So, for $g\in G$, $f(x_{g})\in h(T)$ with $0=f(x)=f\left(\displaystyle\sum_{g\in G}x_{g}\right)=\displaystyle\sum_{g\in G}f(x_{g})$, which yields that $f(x_{g})=0$ for all $g\in G$ as $\{0\}$ is a graded ideal. Therefore, $x_{g}\in Ker(f)$ for all $g\in G$, and hence $Ker(f)$ is a graded ideal of $R$.
\end{proof}

\begin{thm}\label{Theorem 4.10} Let $R$ and $T$ be two $G$-rings and $f:R\rightarrow T$ be a surjective graded homomorphism.
\begin{enumerate}
\item If $P$ is a graded weakly $2$-absorbing ideal of $R$ and $Ker(f)\subseteq P$, then $f(P)$ is a graded weakly $2$-absorbing ideal of $T$.

\item If $I$ is a graded weakly $2$-absorbing ideal of $T$ and $Ker(f)$ is a graded weakly $2$-absorbing ideal of $R$, then $f^{-1}(I)$ is a graded weakly $2$-absorbing ideal of $R$.
\end{enumerate}
\end{thm}

\begin{proof}
\begin{enumerate}
\item By (\cite{Refai Dawwas}, Lemma 3.11 (2)), $f(P)$ is a graded ideal of $T$. Since $P$ is a graded weakly $2$-absorbing ideal of $R$ and $Ker(f)\subseteq P$, we conclude that $P/Ker(f)$ is a graded weakly $2$-absorbing ideal of $R/Ker(f)$ by Proposition \ref{Proposition 4.8}. Since $R/Ker(f)$ is isomorphic to $T$, the result holds.

\item By (\cite{Refai Dawwas}, Lemma 3.11 (1)), $f^{-1}(I)$ is a graded ideal of $R$. Let $K=f^{-1}(I)$. Then $Ker(f)\subseteq K$. Since $R/Ker(f)$ is isomorphic to $T$,we conclude that $K/Ker(f)$ is a graded weakly $2$-absorbing ideal of $R/Ker(f)$. Since $Ker(f)$ is a graded weakly $2$-absorbing ideal of $R$ and $K/Ker(f)$ is a graded weakly $2$-absorbing ideal of $R/Ker(f)$, we conclude that $K=f^{-1}(I)$ is a graded weakly $2$-absorbing ideal of $R$ by Proposition \ref{Proposition 4.9}.
\end{enumerate}
\end{proof}

Motivated by Theorem \ref{Theorem 4.7}, we have the following question:

\begin{question} Suppose that $P$ is a graded weakly $2$-absorbing ideal of $R$ that is not a graded $2$-absorbing ideal and $0\neq ABK\subseteq P$ for some graded ideals $A, B$ and $K$ of $R$. Does it imply that $AB\subseteq P$ or $AK\subseteq P$ or $BK\subseteq P$?
\end{question}

We will give a partial answer through the coming discussions. Motivated by (\cite{Zoubi Dawwas Ceken}, Definition 3.3), we introduce the following:

\begin{defn} Let $R$ be a $G$-graded ring, $g\in G$ and $P$ be a graded ideal of $R$ with $P_{g}\neq R_{g}$.
\begin{enumerate}
\item $P$ is said to be a $g$-$2$-absorbing ideal of $R$ if whenever $x, y, z\in R_{g}$ such that $xR_{e}yR_{e}z\subseteq P$, then $xy\in P$ or $yz\in P$ or $xz\in P$.

\item $P$ is said to be a $g$-weakly $2$-absorbing ideal of $R$ if whenever $x, y, z\in R_{g}$ such that $0\neq xR_{e}yR_{e}z\subseteq P$, then $xy\in P$ or $yz\in P$ or $xz\in P$.

\item Let $P$ be a $g$-weakly $2$-absorbing ideal of $R$ and $x, y, z\in R_{g}$. We say that $(x, y, z)$ is a $g$-triple-zero of $P$ if $xR_{e}yR_{e}z=0$, $xy\notin P$, $yz\notin P$ and $xz\notin P$.
\end{enumerate}
\end{defn}

Note that if $P$ is a $g$-weakly 2-absorbing ideal of $R$ that is not a $g$-$2$-absorbing ideal, then $P$ has a $g$-triple-zero $(x, y, z)$ for some $x, y, z\in R_{g}$.

\begin{prop}\label{Proposition 4.12} Let $P$ be a $g$-weakly $2$-absorbing ideal of $R$ and suppose that $xR_{e}yK_{g}\subseteq P$ for some $x, y\in R_{g}$ and some graded left ideal $K$ of $R$. Assume that $(x, y, z)$ is not a $g$-triple-zero of $P$ for every $z\in K_{g}$. If $xy\notin P$, then $xK_{g}\subseteq P$ or $yK_{g}\subseteq P$.
\end{prop}

\begin{proof} Suppose that $xK_{g}\nsubseteq P$ and $yK_{g}\nsubseteq P$. Then there exist $r, s\in K_{g}$ such that $xr\notin P$ and $ys\notin P$. Since $xR_{e}yR_{e}r\subseteq xR_{e}yK_{g}\subseteq P$ and since $(x, y, r)$ is not a $g$-triple-zero of $P$ and $xy\notin P$, $xr\notin P$, we obtain that $yr\in P$. Also, since $xR_{e}yR_{e}s\subseteq xR_{e}yK_{g}\subseteq P$ and since $(x, y, s)$ is not a $g$-triple-zero of $P$ and $xy\notin P$, $ys\notin P$, we obtain that $xs\in P$. Now, since $xR_{e}yR_{e}(r+s)\subseteq xR_{e}yK_{g}\subseteq P$ and since $(x, y, r+s)$ is not a $g$-triple-zero of $P$ and $xy\notin P$, we get that $x(r+s)\in P$ or $y(r+s)\in P$. If $x(r+s)\in P$, then since $xs\in P$, $xr\in P$, a contradiction. If $y(r+s)\in P$, then since $yr\in P$, $ys\in P$, a contradiction. Hence, $xK_{g}\subseteq P$ or $yK_{g}\subseteq P$.
\end{proof}

\begin{defn} Let $R$ be a $G$-graded ring, $g\in G$ and $P$ be a $g$-weakly $2$-absorbing ideal of $R$ and suppose that $A_{g}B_{g}K_{g}\subseteq P$ for some graded ideals $A, B$ and $K$ of $R$. We say that $P$ is free $g$-triple-zero with respect to $ABK$ if $(x, y, z)$ is not a $g$-triple-zero of $P$ for every $x\in A_{g}, y\in B_{g}$ and $z\in K_{g}$.
\end{defn}

The next proposition is clear.

\begin{prop}\label{Remark 4.14} Let $P$ be a $g$-weakly $2$-absorbing ideal of $R$ and suppose that $A_{g}B_{g}K_{g}\subseteq P$ for some graded ideals $A, B$ and $K$ of $R$ such that $P$ is free $g$-triple-zero with respect to $ABK$. If $x\in A_{g}, y\in B_{g}$ and $z\in K_{g}$, then $xy\in P$ or $xz\in P$ or $yz\in P$.
\end{prop}

\begin{thm}\label{Theorem 4.15} Let $P$ be a $g$-weakly $2$-absorbing ideal of $R$ and suppose that $0\neq A_{g}B_{g}K_{g}\subseteq P$ for some graded ideals $A, B$ and $K$ of $R$ such that $P$ is free $g$-triple-zero with respect to $ABK$. Then $A_{g}K_{g}\subseteq P$ or $B_{g}K_{g}\subseteq P$ or $A_{g}B_{g}\subseteq P$.
\end{thm}

\begin{proof} Suppose that $A_{g}K_{g}\nsubseteq P$, $B_{g}K_{g}\nsubseteq P$ and $A_{g}B_{g}\nsubseteq P$. Then there exist $x\in A_{g}$ and $y\in B_{g}$ such that $xK_{g}\nsubseteq P$ and $yK_{g}\nsubseteq P$. Now, $xR_{e}yK_{g}\subseteq A_{g}B_{g}K_{g}\subseteq P$. Since $xK_{g}\nsubseteq P$ and $yK_{g}\nsubseteq P$, it follows from Proposition \ref{Proposition 4.12} that $xy\in P$. Since $A_{g}B_{g}\nsubseteq P$, there exist $a\in A_{g}$ and $b\in B_{g}$ such that $ab\notin P$. Since $aR_{e}bK_{g}\subseteq A_{g}B_{g}K_{g}\subseteq P$ and $ab\notin P$, it follows from Proposition \ref{Proposition 4.12} that $aK_{g}\subseteq P$ or $bK_{g}\subseteq P$.

\underline{Case (1):} $aK_{g}\subseteq P$ and $bK_{g}\nsubseteq P$. Since $xR_{e}bK_{g}\subseteq A_{g}B_{g}K_{g}\subseteq P$ and $xK_{g}\nsubseteq P$ and $bK_{g}\nsubseteq P$, it follows from Proposition \ref{Proposition 4.12} that $xb\in P$. Since $aK_{g}\subseteq P$ and $xK_{g}\nsubseteq P$, we obtain that $(x+a)K_{g}\nsubseteq P$. On the other hand, since $(x+a)R_{e}bK_{g}\subseteq P$ and neither $(x+a)K_{g}\subseteq P$ nor $bK_{g}\subseteq P$, we have that $(x+a)b\in P$ by Proposition \ref{Proposition 4.12}, and hence $ab\in P$, a contradiction.

\underline{Case (2):} $bK_{g}\subseteq P$ and $aK_{g}\nsubseteq P$. Using an analogous altercation to case (1), we will have a contradiction.

\underline{Case (3):} $aK_{g}\subseteq P$ and $bK_{g}\subseteq P$. Since $bK_{g}\subseteq P$ and $yK_{g}\nsubseteq P$, $(y+b)K_{g}\nsubseteq P$. But $xR_{e}(y+b)K_{g}\subseteq P$ and neither $xK_{g}\subseteq P$ nor $(y+b)K_{g}\subseteq P$, and hence $x(y+b)\subseteq P$ by Proposition \ref{Proposition 4.12}. Since $xy\in P$ and $(xy+xb)\in P$, we have that $xb\in P$. Since $(x+a)R_{e}yK_{g}\subseteq P$ and neither $yK_{g}\subseteq P$ nor $(x+a)K_{g}\subseteq P$, we conclude that $(x+a)y\in P$ by Proposition \ref{Proposition 4.12}, and hence $ax\in P$. Since $(x+a)R_{e}(y+b)K_{g}\subseteq P$ and neither $(x+a)K_{g}\subseteq P$ nor $(y+b)K_{g}\subseteq P$, we have $(x+a)(y+b)\in P$ by Proposition \ref{Proposition 4.12}. But $xy, xb, ay\in P$, so $ab\in P$, a contradiction. Consequently, $A_{g}K_{g}\subseteq P$ or $B_{g}K_{g}\subseteq P$ or $A_{g}B_{g}\subseteq P$.
\end{proof}

\begin{lem}\label{Lemma 4.16 (1)} Let $P$ be a $g$-weakly $2$-absorbing ideal of $R$ and suppose that $(x, y, z)$ is a $g$-triple-zero of $P$ for some $x, y, z\in R_{g}$. Then $xR_{e}yP_{g}=\{0\}$.
\end{lem}

\begin{proof} Suppose that $xR_{e}yP_{g}\neq\{0\}$. Then there exist $r\in R_{e}$ and $p\in P_{g}$ such that $0\neq xryp$. Now, $xry(p+z)=xryp+xryz=xryp\neq0$. Hence, $0\neq xR_{e}yR_{e}(p+z)\subseteq P$. Since $P$ is $g$-weakly $2$-absorbing, we have that $x(p+z)\in P$ or $y(p+z)\in P$. Thus $xz\in P$ or $yz\in P$, a contradiction.
\end{proof}

\begin{lem}\label{Lemma 4.16 (2)} Let $P$ be a $g$-weakly $2$-absorbing ideal of $R$ and suppose that $(x, y, z)$ is a $g$-triple-zero of $P$ for some $x, y, z\in R_{g}$. Then $P_{g}yR_{e}z=\{0\}$.
\end{lem}

\begin{proof} Suppose that $P_{g}yR_{e}z\neq\{0\}$. Then there exist $r\in R_{e}$ and $p\in P_{g}$ such that $0\neq pyrz$. Now, $(x+p)yrz=xyrz+pyrz=pyrz\neq0$. Hence, $0\neq(x+p)R_{e}yR_{e}z\subseteq P$. Since $P$ is $g$-weakly $2$-absorbing, we have that $(x+p)y\in P$ or $(x+p)z\in P$. Hence, $xy\in P$ or $xz\in P$, a contradiction.
\end{proof}

\begin{lem}\label{Lemma 4.16 (3)} Let $P$ be a $g$-weakly $2$-absorbing ideal of $R$ and suppose that $(x, y, z)$ is a $g$-triple-zero of $P$ for some $x, y, z\in R_{g}$. Then $xP_{g}z=\{0\}$.
\end{lem}

\begin{proof} Suppose that $xP_{g}z\neq\{0\}$. Then there exists $p\in P_{g}$ such that $0\neq xpz$. Now, $x(y+p)z=xyz+xpz=xpz\neq0$. Hence, $0\neq xR_{e}(y+p)R_{e}z\subseteq P$. Since $P$ is $g$-weakly $2$-absorbing, we have that $x(y+p)\in P$ or $(y+p)z\in P$. Hence, $xy\in P$ or $yz\in P$, a contradiction.
\end{proof}

\begin{lem}\label{Lemma 4.16 (4)} Let $P$ be a $g$-weakly $2$-absorbing ideal of $R$ and suppose that $(x, y, z)$ is a $g$-triple-zero of $P$ for some $x, y, z\in R_{g}$. Then $P_{g}^{2}z=\{0\}$.
\end{lem}

\begin{proof} Suppose that $P_{g}^{2}z\neq\{0\}$. Then there exist $p, q\in P_{g}$ such that $0\neq pqz$ .Now, $(x+p)(y+q)z=xyz+xqz+pyz+pqz=pqz\neq0$ by Lemma \ref{Lemma 4.16 (2)} and Lemma \ref{Lemma 4.16 (3)}. Hence, $0\neq(x+p)R_{e}(y+q)R_{e}z\subseteq P$. Since $P$ is $g$-weakly $2$-absorbing, we have that $(x+p)z\in P$ or $(y+q)z\in P$ or $(x+p)(y+q)\in P$. Hence, $xz\in P$ or $yz\in P$ or $xy\in P$, a contradiction.
\end{proof}

\begin{lem}\label{Lemma 4.16 (5)} Let $P$ be a $g$-weakly $2$-absorbing ideal of $R$ and suppose that $(x, y, z)$ is a $g$-triple-zero of $P$ for some $x, y, z\in R_{g}$. Then $xP_{g}^{2}=\{0\}$.
\end{lem}

\begin{proof} Suppose that  $xP_{g}^{2}\neq\{0\}$. Then there exist $p, q\in P_{g}$ such that $0\neq xpq$. Now, $x(y+p)(z+q)=xyz+xyq+xpz+xpq=xpq\neq0$ by Lemma \ref{Lemma 4.16 (1)} and Lemma \ref{Lemma 4.16 (3)}. Hence, $0\neq xR_{e}(y+p)R_{e}(z+q)\subseteq P$. Since $P$ is $g$-weakly $2$-absorbing, we have that $x(y+p)\in P$ or $x(z+q)\in P$ or $(y+p)(z+q)\in P$. Hence, $xy\in P$ or $xz\in P$ or $yz\in P$, a contradiction.
\end{proof}

\begin{lem}\label{Lemma 4.16 (6)} Let $P$ be a $g$-weakly $2$-absorbing ideal of $R$ and suppose that $(x, y, z)$ is a $g$-triple-zero of $P$ for some $x, y, z\in R_{g}$. Then $P_{g}yP_{g}=\{0\}$.
\end{lem}

\begin{proof} Suppose that $P_{g}yP_{g}\neq\{0\}$. Then there exist $p, q\in P_{g}$ such that $0\neq pyq$. Now, $(x+p)y(z+q)=xyz+xyq+pyz+pyq=pyq\neq0$ Lemma \ref{Lemma 4.16 (1)} and Lemma \ref{Lemma 4.16 (2)}. Hence, $0\neq(x+p)R_{e}yR_{e}(z+q)\subseteq P$. Since $P$ is $g$-weakly $2$-absorbing, we have that $(x+p)y\in P$ or $y(z+q)\in P$ or $(x+p)(z+q)\in P$. Hence, $xy\in P$ or $yz\in P$ or $xz\in P$, a contradiction.
\end{proof}

The next theorem is a consequence result from Lemma \ref{Lemma 4.16 (1)} - Lemma \ref{Lemma 4.16 (6)}.

\begin{thm}\label{Theorem 4.17} Let $R$ be a $G$-graded ring, $g\in G$ and $P$ be a ideal of $R$ such that $P_{g}^{3}\neq\{0\}$. Then $P$ is a $g$-weakly $2$-absorbing ideal of $R$ if and only if $P$ is a $g$-2-absorbing ideal of $R$.
\end{thm}

\begin{proof} Suppose that $P$ is a $g$-weakly $2$-absorbing ideal which is not a $g$-$2$-absorbing ideal of $R$. Then $P$ has a $g$-triple-zero, say $(x, y, z)$ for some $x, y, z\in R_{g}$. Since $P_{g}^{3}\neq\{0\}$, there exist $p, q, r\in P_{g}$ such that $pqr\neq0$, and then $(x+p)(y+q)(z+r)=pqr\neq0$. Hence, $0\neq(x+p)R_{e}(y+q)R_{e}(z+r)\subseteq P$. Since $P$ is $g$-weakly $2$-absorbing, we have either $(x+p)(y+q)\in P$ or $(x+p)(z+r)\in P$ or $(y+q)(z+r)\in P$, and thus either $xy\in P$ or $xz\in P$ or $yz\in P$, a contradiction. Hence, $P$ is a $g$-$2$-absorbing ideal of $R$. The converse is obvious.
\end{proof}

\begin{cor}\label{Corollary 4.18} If $P$ is a $g$-weakly $2$-absorbing ideal of $R$ that is not a $g$-$2$-absorbing ideal, then $P_{g}^{3}=\{0\}$.
\end{cor}

Let $R\ $be a $G$-graded ring. Then an $R$-module $M$ is said to be
$G$\textit{-graded} if $M=\bigoplus_{g\in G}M_{g}$ with $R_{g}M_{h}\subseteq
M_{gh}$ for all $g,h\in G$, where $M_{g}$ is an additive subgroup of $M$ for
all $g\in G$. The elements of $M_{g}$ are called homogeneous of degree $g$. It
is clear that $M_{g}$ is an $R_{e}$-submodule of $M$ for all $g\in G$. We
assume that $h(M)=\bigcup_{g\in G}M_{g}$. Let $N$ be an $R$-submodule of a
graded $R$-module $M$. Then $N$ is said to be graded $R$-submodule if
$N=\bigoplus_{g\in G}(N\cap M_{g})$, or equivalently, $x=\sum_{g\in G}x_{g}\in
N\ $implies that $x_{g}\in N$ for all $g\in G$. It is known that an
$R$-submodule of a graded $R$-module need not be graded. For more terminology
see \cite{Hazart, Nastasescue}.

Let $M$ be an $R$-bi-module. The idealization $R\ltimes M=\{(r,m):r\in R,m\in
M\}$ of $M$ is a ring with componentwise addition and
multiplication: $(x,m_{1})+(y,m_{2})=(x+y,m_{1}+m_{2})$ and $(x,m_{1}%
)(y,m_{2})=(xy,xm_{2}+m_{1}y)$ for each $x,y\in R$ and $m_{1},m_{2}\in M$. Let
$G$ be an Abelian group and $M$ be a $G$-graded $R$-module. Then $X=R\ltimes
M$ is $G$-graded by $X_{g}=R_{g}\bigoplus M_{g}$ for all $g\in G$ \cite{RaTeShKo}.

\begin{thm}\label{Theorem 5.1} Let $R$ be a graded ring with unity, $M$ be a graded $R$-bi-module and $P$ be a proper graded ideal of $R$. Then $P\ltimes M$ is a graded $2$-absorbing ideal of $R\ltimes M$ if and only if $P$ is a graded 2-absorbing ideal of $R$.
\end{thm}

\begin{proof} Assume that $P\ltimes M$ is a graded $2$-absorbing ideal of $R\ltimes M$ and $xRyRz\subseteq P$ for some $x, y, z\in h(R)$. Then $(x, 0), (y, 0), (z, 0)\in h(R\ltimes M)$ with $(x, 0)R\ltimes M(y, 0)R\ltimes M(z, 0)\subseteq P\ltimes M$, and then $(x, 0)(y, 0)=(xy, 0)\subseteq P\ltimes M$ or $(x, 0)(z, 0)=(xz, 0)\subseteq P\ltimes M$ or $(y, 0)(z, 0)=(yz, 0)\subseteq P\ltimes M$. Hence, $xy\in P$ or $xz\in P$ or $yz\in P$, as required. Conversely, let $(x,m)R\ltimes M(y, n)R\ltimes M(z, p)\subseteq P\ltimes M$ for some $(x,m), (y, n), (z, p)\in h(R\ltimes M)$. Hence, $x, y, z\in h(R)$ with $xRyRz\subseteq P$, and then we have $xy\in P$ or $xz\in P$ or $yz\in P$. If $xy\in P$, then $(x,m)(y, n)=(xy, xn+ym)\subseteq P\ltimes M$. Similarly, if $xz\in P$, then $(x, m)(z, p)\in P\ltimes M$, and if $yz\in P$, then $(y, n)(z, p)\in P\ltimes M$, and this completes the proof.
\end{proof}

\begin{thm}\label{Theorem 5.2 (1)} Let $R$ be a graded ring with unity, $M$ be a graded $R$-bi-module and $P$ be a proper graded ideal of $R$. If $P\ltimes M$ is a graded weakly $2$-absorbing ideal of $R\ltimes M$, then $P$ is a graded weakly 2-absorbing ideal of $R$.
\end{thm}

\begin{proof} Let $0\neq xRyRz\subseteq P$, where $x, y, z\in h(R)$. Then $(0, 0)\neq(x, 0)R\ltimes M(y, 0)R\ltimes M(z, 0)\subseteq P\ltimes M$, and then $(xy, 0)\in P\ltimes M$ or $(xz, 0)\in P\ltimes M$ or $(yz, 0)\in P\ltimes M$. Hence, $xy\in P$ or $xz\in P$ or $yz\in P$. So, $P$ is graded weakly $2$-absorbing.
\end{proof}

\begin{thm}\label{Theorem 5.2 (2)} Let $R$ be a $G$-graded ring with unity, $M$ be a graded $R$-bi-module, $g\in G$ and $P$ be a graded ideal of $R$ with $P_{g}\neq R_{g}$. Then $P\ltimes M$ is a $g$-weakly $2$-absorbing ideal of $R\ltimes M$ if and only if $P$ is a $g$-weakly 2-absorbing ideal of $R$ and for any $g$-triple zero $(x, y, z)$ of $P$ we have $xR_{e}yR_{e}M_{g}=M_{g}R_{e}yR_{e}z=xM_{g}z=0$.
\end{thm}

\begin{proof} Suppose that $P\ltimes M$ is a $g$-weakly $2$-absorbing ideal of $R\ltimes M$. Let $0\neq xR_{e}yR_{e}z\subseteq P$, where $x, y, z\in R_{g}$. Then $(0, 0)\neq(x, 0)R_{e}\ltimes M_{e}(y, 0)R_{e}\ltimes M_{e}(z, 0)\subseteq P\ltimes M$, and then $(xy, 0)\in P\ltimes M$ or $(xz, 0)\in P\ltimes M$ or $(yz, 0)\in P\ltimes M$. Hence, $xy\in P$ or $xz\in P$ or $yz\in P$. So, $P$ is $g$-weakly $2$-absorbing. Suppose that $(x, y, z)$ is a $g$-triple zero of $P$. Assume that $xR_{e}yR_{e}M_{g}\neq0$. Then there exist $r, s\in R_{e}$ and $m\in M_{g}$ such that $xrysm\neq0$, and then $(0, 0)\neq(xrysz, xrysm)=(x, 0)(r, 0)(y, 0)(s, 0)(z,m)\in (x, 0)R_{e}\ltimes M_{e}(y, 0)R_{e}\ltimes M_{e}(z,m)\subseteq xR_{e}yR_{e}z\ltimes M_{g}=0\ltimes M_{g}\subseteq P\ltimes M$. But $(x, 0)(y, 0)\notin P\ltimes M$ and $(x, 0)(z,m)\notin P\ltimes M$ and $(y, 0)(z,m)\notin P\ltimes M$, which contradicts the fact that $P\ltimes M$ is a $g$-weakly $2$-absorbing ideal. If $M_{g}R_{e}yR_{e}z\neq0$, then there exist $n\in M_{g}$ and $r, s\in R_{e}$ such that $nrysz\neq0$. As above, we have $(0, 0)\neq(xrysz, nrysz)=(x, n)(r, 0)(y, 0)(s, 0)(z, 0)\in(x, n)R_{e}\ltimes M_{e}(y, 0)R_{e}\ltimes M_{e}(z, 0)\subseteq xR_{e}yR_{e}z\ltimes M_{g}=0\ltimes M_{g}\subseteq P\ltimes M$. But $(x, n)(y, 0)\notin P\ltimes M$ and $(x, n)(z, 0)\notin P\ltimes M$ and $(y, 0)(z, 0)\notin P\ltimes M$ and again contradicting the fact that $P\ltimes M$ is a $g$-weakly $2$-absorbing ideal. If $xM_{g}z\neq0$, then there exists $t\in M_{g}$ such that $xtz\neq0$. Now, $(0, 0)\neq(xyz, xtz)=(x, 0)(1, 0)(y, t)(1, 0)(z, 0)\in(x, 0)R_{e}\ltimes M_{e}(y, t)R_{e}\ltimes M_{e}(z, 0)\subseteq xR_{e}yR_{e}z\ltimes M_{g}=0\ltimes M_{g}\subseteq P\ltimes M$. But $(x, 0)(y, t)\notin P\ltimes M$ and $(x, 0)(z, 0)\notin P\ltimes M$ and $(y, t)(z, 0)\notin P\ltimes M$ so contradicting the fact that $P\ltimes M$ is a $g$-weakly $2$-absorbing ideal. Conversely, suppose that $(0, 0)\neq(x, n)R_{e}\ltimes M_{e}(y,m)R_{e}\ltimes M_{e}(z, t)\subseteq P\ltimes M$ for $(x, n), (y,m), (z, t)\in R_{g}\ltimes  M_{g}$. Then $x, y, z\in R_{g}$ with $xR_{e}yR_{e}z\subseteq P$.

\underline{Case (1):} $xR_{e}yR_{e}z\neq0$. Since $P$ is $g$-weakly $2$-absorbing, $xy\in P$ or $xz\in P$ or $yz\in P$. Hence, $(x, n)(y,m)\in P\ltimes M$ or $(x, n)(z, t)\in P\ltimes M$ or $(y,m)(z, t)\in P\ltimes M$, as desired.

\underline{Case (2):} $xR_{e}yR_{e}z\neq0$. If $xy\notin P$ and $xz\notin P$ and $yz\notin P$, then $(x, y, z)$ is a $g$-triple zero of $P$ and by assumption $xR_{e}yR_{e}M_{g}=M_{g}R_{e}yR_{e}z=xM_{g}z=0$. Now, $(x, n)R_{e}\ltimes M_{e}(y, m)R_{e}\ltimes M_{e}(z, t)\subseteq\left(xR_{e}yR_{e}z, M_{g}R_{e}yR_{e}z+xM_{g}z+xR_{e}yR_{e}M_{g}\right)=(0, 0)$, a contradiction.
\end{proof}

\begin{defn} Let $R$ be a graded ring and $P$ be a proper graded ideal of $R$. Then P is said to be a graded strongly weakly $2$-absorbing ideal of $R$ if whenever $A, B$ and $C$ are graded ideals of $R$ such that $0\neq ABC\subseteq P$, then $AC\subseteq P$ or $BC\subseteq P$ or $AB\subseteq P$. 
\end{defn}

\begin{prop}\label{Proposition 6.9} Let $P$ be a proper graded ideal of $R$. Then $P$ is a graded strongly weakly $2$-absorbing ideal of $R$ if and only if for any graded ideals $A, B$ and $C$ of $R$ such that $P\subseteq A$, $0\neq ABC\subseteq P$ implies that $AB\subseteq P$ or $AC\subseteq P$ or $BC\subseteq P$. 
\end{prop}

\begin{proof} If $P$ is a graded strongly weakly $2$-absorbing ideal of $R$, then the result holds by definition. Conversely, let $K, B$ and $C$ be graded ideals of $R$ such that $0\neq KBC\subseteq P$. Then $A=K+P$ is a graded ideal of $R$ such that $0\neq ABC\subseteq P$, and then by assumption, $AB\subseteq P$ or $AC\subseteq P$ or $BC\subseteq P$. Hence $KB\subseteq P$ or $KC\subseteq P$ or $BC\subseteq P$. So, $P$ is a graded strongly weakly $2$-absorbing ideal of $R$.
\end{proof}

\begin{prop}\label{Proposition 6.10} Let $R$ be a graded ring. Then every graded ideal of $R$ is graded strongly weakly $2$-absorbing if and only
if for any graded ideals $I, J$ and $K$ of $R$, $IJ=IJK$ or $IK=IJK$ or $JK=IJK$ or $IJK=0$.
\end{prop}

\begin{proof} Suppose that every graded ideal of $R$ is graded strongly weakly $2$-absorbing. Let $I, J$ and $K$ be graded ideals of $R$. If $IJK\neq R$, then $IJK$ is graded strongly weakly $2$-absorbing. Suppose that $IJK\neq0$. Then $0\neq IJK\subseteq IJK$ and $IJ\subseteq IJK$ or $IK\subseteq IJK$ or $JK\subseteq IJK$ and hence $IJ=IJK$ or $IK=IJK$ or $JK=IJK$. If $IJK=R$, then $I=J=K=R$. Therefore, $R=R^{3}$. Conversely, let $P$ be a proper graded ideal of $R$, $0\neq IJK\subseteq P$ for some graded ideals $I, J$ and $K$ of $R$. Then $IJ=IJK\subseteq P$ or $IK=IJK\subseteq P$ or $JK=IJK\subseteq P$. Hence, $P$ is a
graded strongly weakly $2$-absorbing ideal of $R$. 
\end{proof}

\begin{cor}\label{Corollary 6.11} Let $R$ be a graded ring such that every graded ideal of $R$ is graded strongly weakly $2$-absorbing. Then for any graded ideal $I$ of $R$, $I^{3}=I^{2}$ or $I^{3}=0$. 
\end{cor}

\begin{question} As a proposal for future work, we think it will be worthy to study non-commutative graded rings in which every graded ideal is graded weakly $2$-absorbing or graded strongly weakly $2$-absorbing. What the kind of the results that will be achieved? 
\end{question}

\end{document}